\documentclass[12pt,reqno]{amsart}

\newcommand\version{September 8, 2019}

\usepackage{amsmath,amsfonts,amsthm,amssymb,amsxtra}
\usepackage{bbm} % to get \1

\newtheorem{theorem}{Theorem}%[section]
\newtheorem{proposition}[theorem]{Proposition}
\newtheorem{lemma}[theorem]{Lemma}

\theoremstyle{definition}

\theoremstyle{remark}

\newtheorem{remark}[theorem]{Remark}

\newcommand{\1}{\mathbbm{1}}

\renewcommand{\epsilon}{\varepsilon}

\renewcommand{\phi}{\varphi}
\newcommand{\R}{\mathbb{R}}

\newcommand{\Sph}{\mathbb{S}}

\newcommand{\Z}{\mathbb{Z}}

\begin{document}

\title[A note on a theorem of M. Christ --- \version]{A note on a theorem of M. Christ}

\author{Rupert L. Frank}
\address[R. L. Frank]{Mathematisches Institut, Ludwig-Maximilans Univers\"at M\"unchen, The\-resienstr. 39, 80333 M\"unchen, Germany, and Munich Center for Quantum Science and Technology (MCQST), Schellingstr. 4, 80799 M\"unchen, Germany, and Mathematics 253-37, Caltech, Pasa\-de\-na, CA 91125, USA}
\email{r.frank@lmu.de, rlfrank@caltech.edu}

\author{Elliott H. Lieb}
\address[E. H. Lieb]{Departments of Mathematics and Physics, Jadwin Hall, Princeton University, Princeton, NJ 08544, USA}
\email{lieb@princeton.edu}

\begin{abstract}
This note is a supplement to our paper \cite{FrLi}. Recently, M. Christ has derived a deep result concerning  stability of the Riesz rearrangement inequality. We are interested in the special case of this inequality where two of the sets are equal and the third set is a fixed ball. We show that in this special case, Christ's proof extends with only minor changes to the case where characteristic functions are replaced by functions taking values between zero and one.
\end{abstract}

%\subjclass[2000]{Primary 35P15, Secondary 35J10, 47F05}

\maketitle

\renewcommand{\thefootnote}{${}$} \footnotetext{\copyright\, 2019 by
  the authors. This paper may be reproduced, in its entirety, for
  non-commercial purposes.\\
  Partial support through US National Science Foundation grant DMS-1363432 and through German Research Foundation grant EXC-2111 390814868 (R.L.F.) is acknowledged.}

\section{The theorem}

A special case of the Riesz rearrangement inequality states that for sets $E\subset\R^N$ of finite measure and a ball $B\subset\R^N$, centered at the origin,
$$
\frac12 \iint_{E\times E} \1_B(x-y)\,dx\,dy \leq \frac12 \iint_{E^*\times E^*} \1_B(x-y)\,dx\,dy \,,
$$
where $E^*$ is the ball, centered at the origin, of measure $|E^*|=|E|$; see the original work \cite{Ri}, as well as \cite[Thm.~3.7]{LiLo} for a textbook proof and \cite[Sec.~2]{Bu} for a historical discussion. Under the strict admissibility condition $0<|B|^{1/N}/(2|E|^{1/N})<1$, Burchard \cite{Bu} showed that equality is attained if and only if $E$ is a ball (up to sets of measure zero). Recently, Christ \cite{Ch} obtained a remainder term in the inequality which measures the distance of $E$ from being a ball. We emphasize that the results of Riesz, Burchard and Christ are valid for a more general inequality involving three functions. Here we restrict our attention to the case where two of these sets are equal and where the third one is a ball.

Our note concerns a generalization of the above inequality to the case where the set $E$, or rather its characteristic function, is replaced by a function $\rho$ taking values between zero and one. The Riesz rearrangement inequality and the bathtub principle \cite[Thm.~1.14]{LiLo} imply that
$$
\frac12 \iint_{\R^N\times\R^N} \rho(x) \1_B(x-y) \rho(y)\,dx\,dy \leq \frac12 \iint_{E^*\times E^*} \1_B(x-y)\,dx\,dy \,,
$$
where $E^*$ is the ball, centered at the origin, of measure $|E^*|=\int_{\R^N} \rho\,dx$. Our goal is to show that Christ's method of obtaining a remainder term works with minor changes in this more general case of functions $\rho$, which is needed in our paper \cite{FrLi}.

For a function $0\not\equiv\rho\in L^1(\R^N)$ with $0\leq\rho\leq 1$, we set
$$
A[\rho] := \left( 2\|\rho\|_1 \right)^{-1} \inf_{a\in\R^N} \| \rho-\1_{E^*+a} \|_1 \,,
$$
where $E^*$ is the ball, centered at the origin, of measure $|E^*|=\int_{\R^N} \rho\,dx$.

\begin{theorem}\label{christ}
Let $0<\delta\leq 1/2$. Then there is a constant $c_{N,\delta}>0$ such that for all balls $B\subset\R^N$, centered at the origin, and all $\rho\in L^1(\R^N)$ with $0\leq\rho\leq 1$ and
$$
\delta \leq \frac{|B|^{1/N}}{2\, \|\rho\|_1^{1/N}} \leq 1-\delta \,,
$$
one has
$$
\frac12 \iint_{\R^N\times\R^N} \rho(x)\1_B(x-y)\rho(y)\,dx\,dy \leq \frac12\iint_{E^*\times E^*} \1_B(x-y)\,dx\,dy - c_{N,\delta} \|\rho\|_1^2\, A[\rho]^2 \,,
$$
where $E^*$ is the ball, centered at the origin, of measure $|E^*|=\int_{\R^N} \rho\,dx$.
\end{theorem}

Our proof yields (in principle) a computable numerical value of the constant $c_{N,\delta}$. No compactness is involved. However, the value obtain is probably very far from being optimal.

As we mentioned before, in the case where $\rho$ is the characteristic function of a set $E$, this theorem is a special case of a more general theorem of Christ \cite{Ch} which concerns three different sets. Let us explain this in more detail. In the case of three arbitrary sets the remainder involves two translation parameters and a matrix with determinant one. However, using the triangle inequality and the fact that in our case two of the sets coincide and the third one is a ball, one can bound this remainder term in terms of $A[\1_E]$, which involves only a single translation parameter. In this way, Theorem \ref{christ} for characteristic functions follows from the result in \cite{Ch}.

Our theorem for $\rho$ with values between zero and one is a modest extension of this result which is obtained by essentially the same method of proof. A related extension was announced in \cite{ChIl}, again in a three function setting. However, the inequality in \cite{ChIl} is of a somewhat different nature since objects defined on a compact Abelian group are compared with their rearrangements in $\R/\Z$.

Here we provide a proof of Theorem \ref{christ} without claiming any conceptual novelty. We prepared this note for three reasons. First, the proof of Theorem \ref{christ}, as stated, does not appear in the literature. Second, we want to prove our claim that the constant $c_{N,\delta}$ is (in principle) computable and that no compactness is used. And third, we provide a somewhat different and more explicit treatment of the quadratic form which, in some sense, corresponds to the Hessian of the functional under consideration. The latter might be useful in other applications and extensions of Christ's analysis in \cite{Ch}.

We emphasize that the treatment of the Hessian is simpler in our case than in Christ's `three sets' case. We explain this in some more detail in Remark \ref{threesets}.

The overall strategy of Christ's proof in \cite{Ch} and our version of it bears some resemblance with the proof of a quantitative stability theorem for the Sobolev inequality by Bianchi and Egnell \cite{BiEg}, answering a question in \cite{BrLi}; see also \cite{ChFrWe}. This strategy consists of a first step which reduces the assertion to elements close to the set of optimizers and a second step where the inequality is proved close to the set of optimizers by a detailed analysis of the eigenvalues of the Hessian of the corresponding variational problem. Christ's analysis in \cite{Ch} is significantly more involved than this standard strategy and adds an additional step, since the metric in which closeness to the optimizers is measured in the first step ($L^1$ distance) and the metric in which the second step can be carried out (Hausdorff distance) are not equivalent.

\medskip

The ball $B$, centered at the origin, will be fixed throughout the following and we will use the notation
$$
\mathcal I[g,h] := \frac12 \iint_{\R^N\times\R^N} g(x)\,\1_{B}(x-y)\, h(y)\,dx\,dy
$$
and $\mathcal I[g]:=\mathcal I[g,g]$.

\subsection*{Acknowledgements}

The authors are grateful to A. Burchard and M. Christ for helpful remarks.

%%%%%%%%%%%%%%%%%%%%

\section{Proof of the bound for small perturbations}

The core of Theorem \ref{christ} is contained in the following proposition which treats the case where $\rho$ differs from a ball only in a shell around the surface of this ball with relative width of order $A[\rho]$ and where $\rho$ is, in a certain sense, centered at the center of this ball. This corresponds to the arguments in \cite[Sections 6, 7, 10]{Ch}.

\begin{proposition}\label{christproofmain}
For every $0<\delta\leq 1/2$ there are constants $\theta_{N,\delta}>0$, $c_{N,\delta,K}>0$ and $C_{N,\delta,K}<\infty$ such that for all balls $B\subset\R^N$, centered at the origin, and all $\rho\in L^1(\R^N)$ with $0\leq\rho\leq 1$ satisfying
\begin{equation*}
\delta \leq \frac{|B|^{1/N}}{2\,\|\rho\|_1^{1/N}} \leq 1-\delta \,,
\end{equation*}
\begin{equation}
\label{eq:rhocom}
\int_{\R^N} \frac{x_n}{|x|}\, \rho\,dx = 0
\qquad\text{for all}\ n=1,\ldots,N
\end{equation}
and, for the ball $E^*$, centered at the origin, of measure $|E^*|=\int_{\R^N} \rho\,dx$ and $0<\theta\leq\theta_{N,\delta}$,
\begin{equation}
\label{eq:rhosupp}
\1_{(1-\theta)E^*} \leq \rho\leq \1_{(1+\theta)E^*} \,,
\end{equation}
one has
$$
\mathcal I[\rho] \leq \mathcal I[\1_{E^*}] - c_{N,\delta}\, \|\rho\|_1^2\, \left( A[\rho]^2 - C_{N,\delta} \theta^3 \right).
$$
\end{proposition}

For the proof of this proposition we need the following auxiliary result concerning the function
$$
\phi(|x|) := \int_{E^*} \1_B(x-y)\,dy \,.
$$
(Note that the right side only depends on $|x|$.)

\begin{lemma}\label{phi}
Let $0<\delta\leq 1/2$ and let $B,E^*\subset\R^N$ be balls centered at the origin with
\begin{equation}
\label{eq:comparabledelta}
\delta \leq \frac{|B|^{1/N}}{2\,|E^*|^{1/N}} \leq 1-\delta \,.
\end{equation}
Let $R$ be the radius of $E^*$. Then $\phi$ is a non-increasing function of $r$ satisfying
$$
|\phi(r) - \phi(R)| \geq c
\begin{cases}
R^{N-1} |r-R| & \text{if}\ |r-R|\leq \frac12 R \,,\\
\frac12 R^N & \text{if}\ |r-R|> \frac12 R \,,
\end{cases}
$$
and, with $\Gamma:=-R^{-N+1} \phi'(R)$,
$$
|\phi(r) - \phi(R) + \Gamma R^{N-1}(r-R)| \leq C R^{N-2} (r-R)^2
\qquad\text{if}\ |r-R|\leq \tfrac12 R \,.
$$
The constants $c>0$ and $C<\infty$ depend only on $N$ and $\delta$.
\end{lemma}

\begin{proof}[Proof of Lemma \ref{phi}]
We observe that $\phi(|x|)=|B_R(0)\cap B_{\tilde R}(x)|$, where $\tilde R$ denotes the radius of $B$. This implies that $\phi$ is non-decreasing with respect to $|x|$ and has a strictly negative radial derivative in $\{|R-\tilde R|< |x|< R+\tilde R \}$. We also note that assumption \eqref{eq:comparabledelta} means that $\delta \leq \tilde R/(2R) \leq 1-\delta$. Therefore, the set $\{|R-\tilde R|< |x|< R+\tilde R \}$ contains the set $\{ (1-2\delta)R<|x|<(1+2\delta)R\}$ and one can show that the radial derivative of $\phi$ does not exceed a negative constant (depending on $N$ and $\delta$) times $R^{N-1}$ on $\{(1-\delta)R\leq |x|\leq (1+\delta)R\}$. Thus, $|\phi(r)-\phi(R)| \geq c R^{N-1} |r-R|$ if $|r-R|\leq\delta R$. On the other hand, for $r\geq (1+\delta) R$ we have, using monotonicity and the bound we just proved, $\phi(R)-\phi(r)\geq \phi(R)-\phi((1+\delta)R)\geq c R^N \delta$. Similarly, for $r\leq (1-\delta)R$, we have $\phi(r)-\phi(R)\geq\phi((1-\delta)R)-\phi(R)\geq c R^N\delta$. After adjusting the value of the constant $c$ we obtain the bound in the lemma. The second bound in the lemma follows similarly.
\end{proof}

\begin{proof}[Proof of Proposition \ref{christproofmain}]
Throughout the proof, we denote by $R$ the radius of $E^*$.

\emph{Step 1.} 
For $\omega\in\Sph^{N-1}$, let
$$
F^+(\omega) := \int_R^\infty \rho(r\omega) r^{N-1}\,dr \,,
\qquad
F^-(\omega) := \int_0^R (1-\rho(r\omega))r^{N-1}\,dr \,.
$$
We claim that, if \eqref{eq:rhosupp} holds with $\theta\leq1$, then
\begin{align}
\label{eq:comparable}
\|\rho\|_1 \,A[\rho] & \leq C_1 \left( \| F^+ \|_{L^2(\Sph^{N-1})}^2 + \| F^- \|_{L^2(\Sph^{N-1})}^2 \,, \right)^{1/2} \\
\label{eq:comparable1}
\max\{\| F^+ \|_{L^\infty(\Sph^{N-1})},\, \| F^- \|_{L^\infty(\Sph^{N-1})}\} & \leq C_2 \|\rho\|_1 \theta \,,
\end{align}
with $C_1$ depending only on $N$ and $C_2$ depending only on $N$ and $\delta$.

Indeed, we begin by noticing that
\begin{align*}
\| \rho \1_{\R^N\setminus E^*} \|_1 & = \|F^+\|_{L^1(\Sph^{N-1})} \leq |\Sph^{N-1}|^{1/2} \| F^+ \|_{L^2(\Sph^{N-1})} \,, \\
\| (1-\rho)\1_{E^*} \|_1 & = \|F^-\|_{L^1(\Sph^{N-1})} \leq |\Sph^{N-1}|^{1/2} \| F^- \|_{L^2(\Sph^{N-1})} \,, \\
\end{align*}
and therefore
\begin{align*}
2 \|\rho\|_1 \, A[\rho] & \leq \| \rho - \1_{E^*} \|_1 = \| \rho \1_{\R^N\setminus E^*} \|_1 + \| (1-\rho)\1_{E^*} \|_1 \\
& \leq |\Sph^{N-1}| \left( \| F^+ \|_{L^2(\Sph^{N-1})} + \| F^- \|_{L^2(\Sph^{N-1})} \right),
\end{align*}
which proves \eqref{eq:comparable}. 

To prove \eqref{eq:comparable1} we use \eqref{eq:rhosupp} and obtain
\begin{align*}
\| F^+\|_{L^\infty(\Sph^{N-1})} & = \sup_{\omega\in\Sph^{N-1}} \int_R^{(1+\theta)R} \rho(r\omega)\,r^{N-1}\,dr \leq \int_R^{(1+\theta)R} \,r^{N-1}\,dr \\
& = \frac{R^N}{N} \left( (1+\theta)^N - 1 \right) \leq \frac{2^N-1}{N} \theta\, R^N = \frac{2^N-1}{|\Sph^{N-1}|}\,\theta\, \|\rho\|_1
\end{align*}
and, similarly,
$$
\| F^-\|_{L^\infty(\Sph^{N-1})} \leq \frac{2^N-1}{|\Sph^{N-1}|}\,\theta\, \|\rho\|_1 \,.
$$
This completes the proof of \eqref{eq:comparable1}.

\medskip

\emph{Step 1.5.} With $f := \rho - \1_{E^*}$ we have
\begin{equation}
\label{eq:expansionmain}
\mathcal I[\rho] = \mathcal I[\1_{E^*}] + 2\, \mathcal I[\1_{E^*},f] + \mathcal I[f] \,.
\end{equation}
In the following two steps we shall show that
\begin{align}
\label{eq:cinteraction}
2\, \mathcal I[\1_{E^*},f] & \leq - \tfrac12 \Gamma \left(\|F^+\|_{L^2(\Sph^{N-1})}^2 + \|F^-\|_{L^2(\Sph^{N-1})}^2 \right) \notag \\
& \, \quad + C R^{-N} \left(\|F^+\|_{L^3(\Sph^{N-1})}^3 + \|F^-\|_{L^3(\Sph^{N-1})}^3 \right)
\end{align}
and
\begin{equation}
\label{eq:cselfenergy}
\left| \mathcal I[f] - \mathcal Q[F^+-F^-] \right| \leq C \, \theta \max\{\|F^+\|_{L^\infty(\Sph^{N-1})}^2,\,\|F^-\|_{L^\infty(\Sph^{N-1})}^2 \}
\end{equation}
with a constant $C$ depending only on $N$ and $\delta$. Here
$$
\mathcal Q[G]:= \frac12 \iint_{\Sph^{N-1}\times\Sph^{N-1}} G(\omega)\1_B(\omega-\omega')\, G(\omega')\,d\omega\,d\omega' \,.
$$

\medskip

\emph{Step 2.} We prove \eqref{eq:cinteraction}. This is \cite[Lemma 16]{Ch}. Since $f$ has integral zero, we have
$$
2\, \mathcal I[\1_E^*,f] = \int_{\R^N} (\phi-\phi(R))f\,dx
$$
with $\phi$ from Lemma \ref{phi}. Let $f^+:=\rho\1_{\R^N\setminus E^*}$ and $f^-:=(1-\rho)\1_{E^*}$, so that $f=f^+-f^-$ and
$$
2\, \mathcal I[\1_E^*,f] = - \int_{\R^N} (\phi(R)-\phi)f^+ \,dx - \int_{\R^N} (\phi-\phi(R))f^- \,dx \,.
$$
Note that by the support properties of $f^+$ and $f^-$ and by the monotonicity of $\phi$ from Lemma \ref{phi} the integrand in both integrals is non-negative. We concentrate on proving a lower bound on the first integral, the second one being similar. We have
$$
\int_{\R^N} (\phi(R)-\phi)f^+\,dx = \int_{\Sph^{N-1}} \left( \int_R^\infty (\phi(R)-\phi(r)) f^+(r\omega) r^{N-1}\,dr \right) d\omega \,.
$$
By the bathtub principle \cite[Thm.~1.14]{LiLo}, using the monotonicity of $\phi$ and the fact that $0\leq f^+\leq 1$, we have, for each $\omega\in\Sph^{N-1}$,
$$
\int_R^\infty (\phi(R)-\phi(r)) f^+(r\omega) r^{N-1}\,dr \geq \int_R^{R^+(\omega)} (\phi(R)-\phi(r)) r^{N-1}\,dr \,,
$$
where $R^+(\omega)$ is uniquely determined by
$$
\int_R^{R^+(\omega)} r^{N-1}\,dr = \int_R^\infty f^+(r\omega)r^{N-1}\,dr \,.
$$
Recalling the definition of $F^+(\omega)$, we find that
$$
R^+(\omega) = (R^N - N F^+(\omega))^{1/N} \,.
$$
Recall that by Step 1,
\begin{equation}
\label{eq:f+bound}
R^{-N} \|F^+\|_{L^\infty(\Sph^{N-1})} \leq \frac{2^N-1}{N}\, \theta \,,
\end{equation}
and therefore, in particular, $R^+(\omega)\leq (3/2)R$ if $\theta$ does not exceed a small constant depending only on $N$. Therefore we can apply the second order Taylor expansion on $\phi$ from Lemma \ref{phi} and obtain
\begin{align*}
\int_R^{R^+(\omega)} (\phi(R)-\phi(r)) r^{N-1}\,dr & \geq \Gamma R^{N-1} \int_R^{R^+(\omega)} (r-R) r^{N-1}\,dr \\
& \quad - C R^{N-2} \int_R^{R^+(\omega)} (r-R)^2 r^{N-1}\,dr \,.
\end{align*}
Inserting the explicit expression for $R^+(\omega)$ and using the bound \eqref{eq:f+bound} to control higher order terms, it is easy to see that the right side is bounded from below by
$$
\tfrac12 \Gamma F^+(\omega)^2 - C' R^{-N} F^+(\omega)^3 \,.
$$
Integrating this bound with respect to $\omega$, we obtain inequality \eqref{eq:cinteraction}.

\medskip

\emph{Step 3.} We prove \eqref{eq:cselfenergy}. This is \cite[Lemma 17]{Ch}. Note that the first term on the left side of \eqref{eq:cselfenergy} is
$$
\frac12 \iint_{\Sph^{N-1}\times\Sph^{N-1}} \int_0^\infty \int_0^\infty f(r\omega) \1_{\{|r\omega - s\omega|\leq \tilde R\}} f(s\omega')\,r^{N-1}\,dr\,s^{N-1}\,ds \,d\omega\,d\omega'
$$
and second one is
$$
\frac12 \iint_{\Sph^{N-1}\times\Sph^{N-1}} F(\omega) \1_{\{R|\omega-\omega'|<\tilde R\}} F(\omega') \,d\omega\,d\omega' \,,
$$
where we abbreviated $F:=F^+-F^-$ and denoted the radius of $B$ by $\tilde R$.

For all $\omega,\omega'\in\Sph^{N-1}$ with $|R|\omega-\omega'|-\tilde R|\geq 2\theta R$ we have
$$
\1_{\{|r\omega - s\omega|\leq \tilde R\}} = \1_{\{R|\omega-\omega'|<\tilde R\}}
\qquad\text{for all}\ r,s\in \left[(1-\theta)R,(1+\theta)R \right]
$$
and consequently, by integration and the support property \eqref{eq:rhosupp},
$$
\int_0^\infty \int_0^\infty f(r\omega) \1_{\{|r\omega - s\omega|\leq \tilde R\}} f(s\omega')\,r^{N-1}\,dr\,s^{N-1}\,ds = F(\omega) \1_{\{R|\omega-\omega'|<\tilde R\}} F(\omega') \,.
$$
Thus, the difference on the left side of \eqref{eq:cselfenergy} comes at most from the set of $(\omega,\omega')\subset\Sph^{N-1}\times\Sph^{N-1}$ such that  $|R|\omega-\omega'| - \tilde R| <2 \theta R$. Clearly, this set has measure at most $C \theta$, where $C$ depends only on $N$ and $\delta$. Since $\int_0^\infty f^\pm(r\omega)r^{N-1}\,dr \leq \|F^\pm\|_{L^\infty(\Sph^{N-1})}$, we obtain the claimed bound.

\medskip

\emph{Step 4.} We now conclude the proof. It is shown in \cite[Corollary 23]{Ch} that
$$
\mathcal Q[G] \leq A\, \Gamma\, \|G\|_{L^2(\Sph^{N-1})}^2
\qquad\text{with a constant}\ A<\tfrac12
$$
for all $G\in L^2(\Sph^{N-1})$ satisfying
\begin{equation}
\label{eq:appfortho0}
\int_{\Sph^{N-1}} G(\omega)\,d\omega = \int_{\Sph^{N-1}} \omega_n G(\omega)\,d\omega = 0 \,,
\qquad\text{for all}\ n=1,\ldots,N \,.
\end{equation}
Note that if $N=1$ this statement is trivial since \eqref{eq:appfortho0} implies that $G\equiv 0$. We will give an alternative proof of the above inequality in Section \ref{sec:hessian}. Now we apply this inequality to $F=F^+-F^-$ and note that the first condition in \eqref{eq:appfortho0} is satisfied since $|E^*|=\int_{\R^N} \rho\,dx$ and the second collection of conditions is satisfied because of \eqref{eq:rhocom}. We also use $F^\pm\geq 0$ in order to bound
$$
\|F\|_{L^2(\Sph^{N-1})}^2\!= \|F^+\|_{L^2(\Sph^{N-1})}^2 + \|F^-\|_{L^2(\Sph^{N-1})}^2 - 2 \langle F^+,F^-\rangle \leq \|F^+\|_{L^2(\Sph^{N-1})}^2 + \|F^-\|_{L^2(\Sph^{N-1})}^2.
$$
Collecting \eqref{eq:expansionmain}, \eqref{eq:cinteraction} and \eqref{eq:cselfenergy} we obtain
$$
\mathcal I[\rho] \leq \mathcal I[\1_{E^*}] - \left(\tfrac12 - A \right) \Gamma \left( \|F^+\|_{L^2(\Sph^{N-1})}^2 + \|F^-\|_{L^2(\Sph^{N-1})}^2 \right) + \mathcal R
$$
with
$$
\mathcal R := C R^{-N} \!\left(\|F^+\|_{L^3(\Sph^{N-1})}^3 \!+ \|F^-\|_{L^3(\Sph^{N-1})}^3 \right) + C \theta \max\left\{\|F^+\|_{L^\infty(\Sph^{N-1})}^2,\, \|F^-\|_{L^\infty(\Sph^{N-1})}^2 \right\}\!.
$$
Using \eqref{eq:comparable1}, we bound
\begin{align*}
\|F^\pm \|_{L^3(\Sph^{N-1})}^3 & \leq \|F^\pm\|_{L^\infty(\Sph^{N-1})} \|F^\pm\|_{L^2(\Sph^{N-1})}^2 \leq C_2 \|\rho\|_1 \theta \|F^\pm\|_{L^2(\Sph^{N-1})}^2 \\
& = C_2' R^N \theta \|F^\pm\|_{L^2(\Sph^{N-1})}^2
\end{align*}
and
$$
\|F^\pm \|_{L^\infty(\Sph^{N-1})}^2 \leq C_2^2\, \theta^2\, \|\rho\|_1^2 \,.
$$

To summarize, we have shown that
$$
\mathcal I[\rho] \leq \mathcal I[\1_{E^*}] - \left(\left(\tfrac12 - A \right) \Gamma - C\,C_2'\, \theta \right) \left( \|F^+\|_{L^2(\Sph^{N-1})}^2 + \|F^-\|_{L^2(\Sph^{N-1})}^2 \right) + C C_2^2\theta^3 \|\rho\|_1^2 \,.
$$
We can bound $\left(\tfrac12 - A \right) \Gamma\geq\tau>0$, where $\tau>0$ depends only on $N$ and $\delta$. This follows, for instance, from the expressions of $\Gamma$ and $A$ in Section \ref{sec:hessian}. Thus, if $\theta\leq \tau/(2CC_2')$, then we have a remainder $\|F^+\|_{L^2(\Sph^{N-1})}^2 + \|F^-\|_{L^2(\Sph^{N-1})}^2$, which, by \eqref{eq:comparable}, is comparable with $\|\rho\|_1^2 \, A[\rho]^2$. This proves the proposition.
\end{proof}

%%%%%%%%%%%%%%%%%%%%%

\section{Reduction to the case of small perturbations}

The following proposition, analogous to \cite[Sec.~5]{Ch}, proves Theorem \ref{christ} for all $\rho$ with $A[\rho]$ small enough. The argument proceeds by reduction to Proposition \ref{christproofmain}.

\begin{proposition}\label{christproofreduction}
For every $0<\delta\leq 1/2$ there are constants $\alpha_{N,\delta}>0$ and $c_{N,\delta}>0$ such that for all balls $B\subset\R^N$, centered at the origin, and all $\rho\in L^1(\R^N)$ with $0\leq\rho\leq 1$,
\begin{equation*}
\delta \leq \frac{|B|^{1/N}}{2\,\|\rho\|_1^{1/N}} \leq 1-\delta
\qquad\text{and}\qquad
A[\rho]\leq\alpha_{N,\delta}
\end{equation*}
one has
$$
\mathcal I[\rho] \leq \mathcal I[\1_{E^*}] - c_{N,\delta} \, \|\rho\|_1^2\, A[\rho]^2 \,,
$$
where $E^*$ is the ball, centered at the origin, of measure $|E^*|=\int_{\R^N}\rho\,dx$.
\end{proposition}

To deduce this proposition from Proposition \ref{christproofmain} we need to argue that the support condition \eqref{eq:rhosupp} and the centering condition \eqref{eq:rhocom} can be imposed. The necessary preliminaries are given in the following two lemmas.

\begin{lemma}\label{competitor}
Let $\rho\in L^1(\R^N)$ with $0\leq\rho\leq 1$. Let $B$ be a ball of measure $\int_{\R^N}\rho\,dx$ and let $0\leq\theta\leq 1$. Then there is a $\tilde\rho\in L^1(\R^N)$ with
\begin{equation}
\label{eq:rhotilde1}
\int_{\R^N}\rho'\,dx = \int_{\R^N} \rho\,dx \,,
\end{equation}
\begin{equation}
\label{eq:rhotilde2}
\1_{(1-\theta)B} \leq \tilde \rho \leq \1_{(1+\theta)B} \,,
\end{equation}
\begin{equation}
\label{eq:rhotilde3}
\tilde \rho \geq \rho
\quad\text{in}\ B
\qquad\text{and}\qquad
\tilde \rho \leq \rho
\quad\text{in}\ \R^N\setminus B \,,
\end{equation}
\begin{equation}
\label{eq:rhotilde5}
\int_{\R^N} |\tilde\rho-\1_B| \,dx \leq \int_{\R^N} |\rho-\1_B|\,dx
\end{equation}
and
\begin{equation}
\label{eq:rhotilde4}
\int_{(1-\theta)B \cup(\R^N\setminus (1+\theta)B)} |\tilde\rho-\rho| \,dx \geq \frac12 \int_{\R^N} |\tilde\rho-\rho|\,dx \,.
\end{equation}
\end{lemma}

This lemma is an extension of the procedure for characteristic functions described in \cite[Sec.~5]{Ch} to functions with values between zero and one. It appears in \cite{FrLi}, but we repeat the proof for the sake of completeness.

\begin{proof}
By translation and scale invariance, we may assume that $B$ is the ball of radius $1$ centered at the origin. Let
$$
m_i:= \int_{\{|x|<1-\theta\}} (1-\rho)\,dx
\qquad\text{and}\qquad
m_o:= \int_{\{|x|>1+\theta\}} \rho\,dx \,.
$$

If $m_i\geq m_o$, we choose $r_o$ such that
$$
\int_{\{|x|>r_o\}} \rho\,dx = m_i
$$
and note that $r_o\leq 1+\theta$. On the other hand, $r_0\geq 1$ since, using the fact that $\rho-\1_B$ has integral zero,
$$
\int_{\{|x|>1\}} \rho\,dx = \int_{\R^N} (\rho-\1_B)_+\,dx = \int_{\R^N} (\rho-\1_B)_-\,dx = \int_B (1-\rho)\,dx \geq m_i \,.
$$
We set
$$
\tilde\rho := \rho\1_{\{|x|\leq r_o\}} + (1-\rho)\1_{\{|x|\leq 1-\theta\}} \,.
$$

If $m_i<m_o$, we choose $r_i$ such that
$$
\int_{\{|x|<r_i\}} (1-\rho)\,dx = m_o
$$
and note that $r_i\geq 1-\theta$. On the other hand, $r_i\leq 1$ by the same computation that showed $r_o\geq 1$. We set
$$
\tilde\rho := \rho\1_{\{|x|\leq 1+\theta\}} + (1-\rho)\1_{\{|x|\leq r_i\}} \,.
$$

In both cases, the properties \eqref{eq:rhotilde1}, \eqref{eq:rhotilde2} and \eqref{eq:rhotilde3} follow immediately from the construction. Moreover, property \eqref{eq:rhotilde5} follows immediately from \eqref{eq:rhotilde3}. In order to prove \eqref{eq:rhotilde5} we set
$$
A := \{ 1-\theta \leq |x| <1+\theta\} \,,
$$
so \eqref{eq:rhotilde5} is equivalent to
$$
\int_{\R^N\setminus A} |\tilde\rho-\rho|\,dx \geq \int_A |\tilde\rho-\rho|\,dx \,.
$$
To unify the treatment of the two cases we set $\rho_i=1-\theta$ if $m_i\geq m_o$ and $\rho_o=1+\theta$ if $m_i<m_o$, so that in both cases
$$
\tilde\rho = \1_{\{|x|\leq r_i \}} + \rho \1_{\{r_i<|x|\leq r_o\}} \,.
$$
Thus,
$$
\int_A |\tilde\rho-\rho|\,dx = \int_{\{1-\theta\leq|x|< r_i\}}(1-\rho)\,dx + \int_{\{r_o\leq |x|< 1+\theta\}} \rho\,dx
$$
We claim that
$$
\int_A |\tilde\rho-\rho|\,dx \leq \max\{m_i,m_o\} \,.
$$
Indeed, if $m_i\geq m_o$, then the set $\{1-\theta\leq|x|< r_i\}$ is empty and
$$
\int_{\{r_o\leq |x|< 1+\theta\}} \rho\,dx = m_o - \int_{\{|x|\geq 1+\theta\}} \rho\,dx \leq m_0 = \max\{m_i,m_o\} \,,
$$
and similarly if $m_i<m_o$. On the other hand,
$$
\int_{\R^N\setminus A} |\tilde\rho-\rho|\,dx = \int_{\{|x|<1-\theta\}}(1-\rho)\,dx + \int_{\{|x|\geq 1+\theta\}} \rho\,dx = m_i + m_o \geq \max\{m_i,m_o\} \,,
$$
as claimed. This completes the proof.
\end{proof}

The second ingredient in the proof of Proposition \ref{christproofreduction} is the following lemma, which shows that the centralization condition \eqref{eq:rhocom} can be imposed. It is the analogue of \cite[Lemma 26]{Ch}.

\begin{lemma}\label{com}
There are $\theta_0>0$ and $C<\infty$ such that for all $0<\theta\leq\theta_0$, all $\rho\in L^1(\R^N)$ with $0\leq\rho\leq 1$ satisfying
$$
\1_{(1-\theta)E^*} \leq\rho\leq \1_{(1+\theta)E^*} \,,
$$
where $E^*\subset\R^N$ is the ball, centered at the origin, with
$$
\int_{\R^N} \rho\,dx = |E^*| \,,
$$
there is an $a\in\R^N$ such that
\begin{align}
\label{eq:com1}
\int_{\R^N} \frac{x_n}{|x|}\,\rho(x+a)\,dx & = 0 \qquad\text{for all}\ n=1,\ldots,N \,,\\
\label{eq:com2}
|a| & \leq C\|\rho\|_1^{1/N} \theta \,,\\
\label{eq:com3}
\1_{(1-C\theta)E^*} & \leq\rho(\cdot+a)\leq \1_{(1+C\theta)E^*} \,. 
\end{align}
\end{lemma}

\begin{proof}[Proof of Lemma \ref{com}]
By scaling it suffices to prove this lemma in the case where $\|\rho\|_1=1$, so $E^*$ is the unit ball. We need to prove that the vector field
$$
F(a) := - \int_{\R^N} \frac{x}{|x|}\,\rho(x+a)\,dx
$$
has a zero in a ball of radius of the order of $\theta$. By dominated convergence, $F$ is continuous. Moreover, we claim that there are $c_0>0$ and $C_0<\infty$ such that, if $\|\rho-\1_{E^*}\|_1 \leq c_0$, then $a\cdot F(a)\geq 0$ for all $|a|=C_0\|\rho-\1_{E^*}\|_1$. By Brouwer's fixed point theorem, this implies that $F$ has a zero in the ball $\{|a|\leq C_0\|\rho-\1_{E^*}\|_1\}$. Since
$$
\|\rho-\1_{E^*}\|_1 \leq |E^*\setminus (1-\theta)E^*| + |(1+\theta)E^*\setminus E^*| = \left( 1- (1-\theta)^N \!\!+ (1+\theta)^N \!\!- 1 \right)|E^*| \leq C \theta,
$$
this implies \eqref{eq:com1} and \eqref{eq:com2} in the statement of the lemma with $\theta_0= c_0/C$. Property \eqref{eq:com3} is a consequence of \eqref{eq:com2} and the corresponding property of $\rho$.

It remains to prove the above claim. To do so, we set $f:=\rho-\1_{E^*}$, so that
\begin{equation}
\label{eq:comproof}
F(a) = -\int_{\R^N} \frac{x}{|x|} \left(\1_{E^*}(x+a) - \1_{E^*}(x)\right)dx - \int_{\R^N} \frac{x}{|x|}f(x+a)\,dx \,.
\end{equation}
We assume in what follows that $|a|\leq 1/2$, so $\1_{E^*}(\cdot+a) - \1_{E^*}(\cdot)$ is supported in $\{||x|-1|\leq 1/2\}$. Therefore, if $g_1,\ldots,g_N$ are $C^2$ functions which coincide with $x_1/|x|,\ldots,x_N/|x|$ on $\{||x|-1|\leq 1/2\}$, then
\begin{align*}
\int_{\R^N} \frac{x_n}{|x|} \left(\1_{E^*}(x+a) - \1_{E^*}(x)\right)dx & = \int_{\R^N} g_n(x) \left(\1_{E^*}(x+a) - \1_{E^*}(x)\right)dx \\
& = \int_{E^*} \left( g_n(x-a)-g_n(x) \right)dx \\
& = - a \cdot \int_{E^*} \nabla g_n(x)\,dx + \mathcal O(a^2) \,.
\end{align*}
A straightforward computation gives
\begin{align*}
a \cdot \int_{E^*} \nabla g_n(x)\,dx = a \cdot \int_{\partial E^*} x g_n(x)\,d\sigma(x) =
a \cdot \int_{\partial E^*} x x_n\,d\sigma(x) = a_n \frac{|\Sph^{N-1}|}{N} \,.
\end{align*}
Inserting this bound into the form \eqref{eq:comproof} of $F$ we obtain
$$
a\cdot F(a) \geq \frac{|\Sph^{N-1}|}{N} |a|^2 - \left( C' |a|^3 + |a| \|f\|_1 \right).
$$
In particular, if $|a|=C_0\|f\|_1$ for a constant $C_0$ to be determined, then
$$
a\cdot F(a) = C_0 \|f\|_1^2 \left( \frac{|\Sph^{N-1}|\,C_0}{N} - \left( C' C_0^2 \|f\|_1 + 1 \right)\right).
$$
Choosing $C_0=2N/|\Sph^{N-1}|$ and $\|f\|_1\leq (C'C_0^2)^{-1}=:c_0$, we obtain $a\cdot F(a)\geq 0$, as claimed. This completes the proof of the lemma.
\end{proof}

Finally, we can prove the main result of this section.

\begin{proof}[Proof of Proposition \ref{christproofreduction}]
As before, we denote the radius of $E^*$ by $R$.

\emph{Step 1.} We assume that $A[\rho]<1$. Then it is easy to see that the infimum defining $A[\rho]$ is attained and after a translation, if necessary, we may assume that it is attained at $a=0$. Thus,
$$
\| \rho - \1_{E^*} \|_1 = 2 \,\|\rho\|_1\, A[\rho]  \,. 
$$
We now apply Lemma \ref{competitor} with the ball $E^*$ and a parameter $0\leq\theta\leq 1/2$ to be determined. We obtain a $\tilde \rho\in L^1(\R^N)$ satisfying \eqref{eq:rhotilde1}, \eqref{eq:rhotilde2}, \eqref{eq:rhotilde3}, \eqref{eq:rhotilde5} and \eqref{eq:rhotilde4} with $B$ replaced by $E^*$. We write
\begin{align}\label{eq:expansionred}
\mathcal I[\rho] - \mathcal I[\tilde\rho] & = \mathcal I[\rho-\tilde\rho,\rho+\tilde\rho] = 2\, \mathcal I[\rho-\tilde\rho,\1_{E^*}] + \mathcal I[\rho-\tilde\rho,\rho+\tilde\rho - 2\cdot\1_{E^*}] \,.
\end{align}

We begin with the first term on the right side of \eqref{eq:expansionred}, which is the main term. Using the fact that $\rho-\tilde\rho$ has integral zero, as well as the monotonicity of $\phi$ from Lemma \ref{phi} and property \eqref{eq:rhotilde2}, we obtain
$$
2\, \mathcal I[\rho-\tilde\rho,\1_{E^*}] = \int_{\R^N} (\rho-\tilde\rho)(\phi-\phi(R))\,dx = - \int_{\R^N} |\rho-\tilde\rho| |\phi-\phi(R)|\,dx \,.
$$
Using the bound from Lemma \ref{phi} (recalling that $\theta\leq 1/2$) as well as \eqref{eq:rhotilde4} we obtain
$$
\int_{\R^N} |\rho-\tilde\rho| |\phi-\phi(R)|\,dx \geq c \theta R^N \int_{||x|-R|\geq\theta R} |\rho-\tilde\rho|\,dx \geq \frac{c}{2} \theta R^N \|\tilde\rho - \rho\|_1 \,.
$$
We now bound the second term on the right side of \eqref{eq:expansionred}. By \eqref{eq:rhotilde5} we obtain
$$
\mathcal I[\rho-\tilde\rho,\rho+\tilde\rho - 2\cdot\1_{E^*}]
\leq \tfrac12 \|\rho-\tilde\rho\|_1 \|\rho+\tilde\rho- 2\cdot 1_{E^*}\|_1 \leq \|\rho-\tilde\rho\|_1 \|\rho-\1_{E^*}\|_1 \,.
$$
Inserting these two bounds in \eqref{eq:expansionred}, we have
\begin{align*}
\mathcal I[\rho] - \mathcal I[\tilde\rho] & \leq - \frac c2 \theta R^N \|\tilde\rho - \rho\|_1 + \|\rho-\tilde\rho\|_1 \|\rho-\1_{E^*}\|_1 \\
& = - \left( \frac{N\,c\,\theta}{2\,|\Sph^{N-1}|} - 2\,A[\rho] \right) \|\rho\|_1\, \|\rho-\tilde\rho\|_1 
\,.
\end{align*}
We now assume that
$$
A[\rho] \leq \frac{N\,c}{12\,|\Sph^{N-1}|}
$$
and choose
$$
\theta := \frac{6\,|\Sph^{N-1}|}{N\,c} \,A[\rho] \leq \frac12 \,. 
$$
Thus, the above inequality becomes
\begin{equation}
\label{eq:expansionred2}
\mathcal I[\rho] - \mathcal I[\tilde\rho] \leq - A[\rho] \, \|\rho\|_1\, \|\rho-\tilde\rho\|_1  \,.
\end{equation}

We now distinguish two cases according to the relative size of $\|\rho-\tilde\rho\|_1$ and $\|\rho-\1_{E^*}\|_1$. If $\|\rho-\tilde\rho\|_1 \geq \frac12 \|\rho-\1_{E^*}\|_1$, then inequality \eqref{eq:expansionred2} implies
$$
\mathcal I[\rho] - \mathcal I[\tilde\rho] \leq - \tfrac12 \, A[\rho] \, \|\rho\|_1\, \|\rho-\1_{E^*} \|_1 = - \|\rho\|_1^2 A[\rho]^2
$$
and we are done since, by Riesz, $\mathcal I[\tilde\rho]\leq \mathcal I[\1_{E^*}]$.

We are left with dealing with the case $\|\rho-\tilde\rho\|_1 < \frac12 \|\rho-\1_{E^*}\|_1 = \|\rho\|_1\, A[\rho]$. In this case we have, for any $a\in\R^N$,
$$
\|\tilde\rho - \1_{E^*+a}\|_1 \geq \|\rho-\1_{E^*+a}\|_1 - \|\rho-\tilde\rho\|_1 \geq 2 \, \|\rho\|_1 \, A[\rho] - \|\rho\|_1 \, A[\rho] = \|\rho\|_1 \, A[\rho] \,,
$$
and therefore
\begin{equation}
\label{eq:acompetitor}
A[\tilde\rho] \geq \frac12\, A[\rho] \,.
\end{equation}
Similarly,
$$
\|\tilde\rho - \1_{E^*}\|_1 \leq \|\rho-\1_{E^*}\|_1 + \|\rho-\tilde\rho\|_1 \leq 2 \|\rho\|_1 \, A[\rho] + \|\rho\|_1 \, A[\rho] = 3\, \|\rho\|_1 \, A[\rho] \,,
$$
and therefore
\begin{equation}
\label{eq:acompetitor2}
A[\tilde\rho]\leq \frac32\, A[\rho] \,.
\end{equation}

By \eqref{eq:acompetitor} and property \eqref{eq:rhotilde2} of $\tilde\rho$,
$$
\1_{(1-K A[\tilde\rho])E^*} \leq \tilde\rho \leq \1_{(1+K A[\tilde\rho])E^*}
\qquad\text{with}\qquad
K := 2\theta A[\rho]^{-1} = \frac{12\,|\Sph^{N-1}|}{N\,c} \,.
$$

Next, we apply Lemma \ref{com} to $\tilde\rho$. In order to apply this lemma, we need to assume that $K A[\tilde\rho]\leq\theta_0$, which, by \eqref{eq:acompetitor2}, is guaranteed if we assume $A[\rho]\leq (2/(3K))\theta_0$, as we may do. Therefore Lemma \ref{com} provides an $a\in\R^N$ such that
\begin{align}
\label{eq:com1appl}
\int_{\R^N} \frac{x_n}{|x|}\,\tilde\rho(x+a)\,dx & = 0 \qquad\text{for all}\ n=1,\ldots,N \,,\\
\label{eq:com3appl}
\1_{(1-C K A[\tilde\rho(\cdot +a)])E^*} & \leq \tilde\rho(\cdot+a)\leq \1_{(1+C K A[\tilde\rho(\cdot+a)])E^*} \,. 
\end{align}
In the last equation, we used the fact that $A[\rho(\cdot+a)]=A[\rho]$. 

Applying Proposition \ref{christproofmain} to $\tilde\rho(\cdot+a)$, we obtain $c_{N,\delta}>0$ such that
$$
\mathcal I[\tilde\rho(\cdot+a)] \leq \mathcal I[\1_{E^*}] - \frac{c_{N,\delta}}2 \| \tilde\rho\|_1^2\, A[\tilde\rho(\cdot+a)]^2
$$
provided that $K A[\tilde\rho(\cdot+a)]\leq\theta_{N,\delta}$ and $C_{N,\delta} K^3 A[\tilde\rho(\cdot +a)]\leq 1/2$. (The last condition is needed to absorb the term $C_{N,\delta} \theta^3$ in Proposition \ref{christproofmain}.) According to \eqref{eq:acompetitor2}, these conditions are satisfied if $A[\rho]\leq (2/3) \min\{\theta_{N,\delta}/K,(2C_{N,\delta} K^3)^{-1}\}$. Using $\mathcal I[\rho]\leq\mathcal I[\tilde\rho]=\mathcal I[\tilde\rho(\cdot+a)]$ from \eqref{eq:expansionred2} as well as \eqref{eq:acompetitor}, we obtain
$$
\mathcal I[\rho] \leq \mathcal I[\1_{E^*}] - \frac{c_{N,\delta}}8 \|\rho\|_1^2\, A[\rho]^2 \,,
$$
which is the claimed inequality.
\end{proof}

%%%%%%%%%%%%%%%%%%%%%%%%%%%%

\section{Proof of Theorem \ref{christ}}

Given the two propositions from the two previous sections there are several ways to complete the proof of Theorem \ref{christ}. The quickest way is probably by using compactness, namely if $(\rho_n)$ is a sequence with $0\leq\rho_n\leq 1$ and $\int_{\R^N} \rho_n\,dx \to |E^*|$ and $\mathcal I[\rho_n]\to \mathcal I[\1_{E^*}]$, then $A[\rho_n]\to 0$. (Note that, in order to get the uniformity in $|B|$, one needs to rescaled the original sequence to achieve that $|B|$ is constant.) The drawback of this method is that, even in principle, it does not lead to an explicit value of the constant $c_{N,\delta}$ in the theorem. Another way to finish the proof of Theorem \ref{christ} would be to extend the flow from \cite[Prop.~11]{Ch}, which continuously transforms a set into a ball of the same volume, from sets to densities $\rho$ with $0\leq\rho\leq 1$. This is somewhat technical. Instead we choose a third way, which extends an argument in \cite{ChIl}.

\begin{proof}[Proof of Theorem \ref{christ}]
In view of Proposition \ref{christproofreduction}, it suffices to prove the inequality for $\rho$ with $A[\rho]>\alpha_{N,\delta}$ for the constant $\alpha_{N,\delta}>0$ from that proposition. Let
$$
\mathcal D[\rho] := \mathcal I[\1_{E^*}] - \mathcal I[\rho] \,.
$$

\emph{Step 1.} We first prove the inequality in the case where $\rho=\1_E$ is the characteristic function of a set $E$. Here we simply repeat the argument from \cite[Sec.~4]{Ch}. Namely, by \cite[Prop.~11]{Ch} there is an $L^1$-continuous, measure-preserving set-valued function $[0,1]\ni t\mapsto E(t)$ with $E(0)=E$ and $E(1)=E^*$. Since $t\mapsto A[\1_{E(t)}]$ is continuous, there is a $t_0\in (0,1)$ such that $A[E(t_0)]=\alpha_{N,\delta}$. Since $t\mapsto \mathcal I[\1_{E(t)}]$ is non-decreasing and by Proposition \ref{christproofreduction} and the fact that $A[\1_E]\leq 1$, we have
\begin{align*}
\mathcal D[\1_E] & = \mathcal I[\1_{E^*}] - \mathcal I[\1_{E(0)}] \geq \mathcal I[\1_{E^*}] - \mathcal I[\1_{E(t_0)}] \geq c_{N,\delta} |E(t_0)|^2 A[\1_{E(t_0)}]^2 \\
& = c_{N,\delta} \alpha_{N,\delta}^2 |E|^2 \geq c_{N,\delta} \alpha_{N,\delta}^2 |E|^2 \, A[\1_E]^2 \,.
\end{align*}
This is the claimed inequality.

\medskip

\emph{Step 2.} We show that there is a constant $c'_{N,\delta}>0$ such that for any $B$ and $\rho$ as in the theorem there is a set $E\subset\R^N$ with $|E|=\int_{\R^N}\rho\,dx$ such that
$$
\mathcal D[\rho] \geq c'_{N,\delta} \|\rho-\1_E \|_1^2 \,.
$$

Indeed, by the Riesz inequality and the bathtub principle \cite[Thm.~1.14]{LiLo} we have $\mathcal I[\rho]\leq\mathcal I[\rho^*] \leq \mathcal I[\1_{E^*},\rho^*]$ and therefore, similarly as in the proof of Proposition \ref{christproofreduction},
\begin{align*}
\mathcal D[\rho] & \geq \mathcal I[\1_{E^*},\1_{E^*}-\rho^*] = \int_{\R^N} (\1_{E^*}-\rho^*)(\phi-\phi(R))\,dx  = \!\int_{\R^N} |\1_{E^*}-\rho^*||\phi-\phi(R)|\,dx
\end{align*}
where $R$ is, as before, the radius of the ball $E^*$ of measure $|E^*|=\int_{\R^N}\rho\,dx$ and where $\phi$ is as in Lemma \ref{phi}. Using the bound on $\phi$ from that lemma we obtain
$$
\mathcal D[\rho] \geq c R^{N-1} \int_{||x|-R|\leq R/2} ||x|-R| |\1_{E^*}-\rho^*|\,dx + \frac c2 R^N \int_{||x|-R|>R/2} |\1_{E^*}-\rho^*|\,dx \,.
$$
Using the one-dimensional inequality $\|\psi\|_\infty \||\cdot|\psi\|_1 \geq (1/4)\|\psi\|_1^2$ we obtain
$$
\int_{||x|-R|\leq R/2} ||x|-R| |\1_{E^*}-\rho^*|\,dx \geq \frac{1}{4} \frac{1}{(3R/2)^{N-1}} \left( \int_{||x|-R|\leq R/2} |\1_{E^*}-\rho^*|\,dx \right)^2 \,.
$$
Thus, the above lower bound on $\mathcal D[\rho]$ becomes
$$
\mathcal D[\rho] \geq \frac{2^{N-1}\, c}{4\cdot 3^{N-1}}\, \|\rho^*-\1_{E^*}\|_1^2 \left( \theta^2 + a (1-\theta) \right)
$$
with
$$
\theta := \frac{\int_{||x|-R|\leq R/2} |\rho^*-\1_{E^*}|\,dx}{\int_{\R^N} |\rho^*-\1_{E^*}|\,dx} 
\qquad\text{and}\qquad
a := \frac{3^{N-1}}{2^{N-2}} \ \frac{R^N}{\|\rho^*-\1_{E^*}\|_1} \,.
$$
Note that $\theta\in[0,1]$. Since for any $b\geq 0$,
$$
\inf_{0\leq t\leq 1} \left( t^2 + b(1-t) \right) =
\begin{cases}
b(1-b/4) & \text{if}\ b<2 \,,\\
1 & \text{if}\ b\geq 2 \,,
\end{cases}
$$
and since
$$
a = \frac{3^{N-1}}{2^{N-2}} \ \frac{N}{|\Sph^{N-1}|}\, \frac{|E^*|}{\|\rho^*-\1_{E^*}\|_1} \geq \frac{3^{N-1}}{2^{N-1}} \ \frac{N}{|\Sph^{N-1}|} \,.
$$
we obtain
$$
\mathcal D[\rho] \geq c_{N,\delta}' \|\rho^*-\1_{E^*}\|_1^2
$$
with some constant $c_{N,\delta}'>0$ depending only on $N$ and $\delta$. Since $\rho^*$ is a rearrangement of $\rho$, there is a set $E\subset\R^N$ with $|E|=|E^*|=\int_{\R^N}\rho\,dx$ such that $\|\rho^*-\1_{E^*}\|_1 = \|\rho - \1_E\|_1$. This proves the claimed inequality.

\medskip

\emph{Step 3.} We prove that the set $E$ from Step 2 satisfies
$$
A[\1_E] \leq c_{N,\delta}^{-1/2} \|\rho\|_1^{-1} \left( \mathcal D[\rho] + \|\rho\|_1 \, \|\rho-\1_E\|_1 \right)^{1/2},
$$
where $c_{N,\delta}$ is the constant such that Theorem \ref{christ} holds for characteristic functions (which has been proved in Step 1). 

Indeed, we have
\begin{align*}
\mathcal D[\1_E] & = \mathcal D[\rho] + \mathcal I[\rho] - \mathcal I[\1_E] = \mathcal D[\rho] + \mathcal I[\rho+\1_E,\rho-\1_E] \leq \mathcal D[\rho] + \|\rho\|_1 \|\rho-\1_E\|_1 \,.
\end{align*}
Combining this with the lower bound on $\mathcal D[\1_E]$ from Theorem \ref{christ} we obtain the claimed bound.

\medskip

\emph{Step 4.} Let us complete the proof of Theorem \ref{christ}. By Step 3 we have
\begin{align*}
A[\rho] & \leq (2\|\rho\|_1)^{-1} \|\rho - \1_E\|_1 + A[\1_E] \\
& \leq \|\rho\|_1^{-1} \left( 2^{-1} \|\rho-\1_E\|_1 + c_{N,\delta}^{-1/2} \left( \mathcal D[\rho] + \|\rho\|_1 \, \|\rho-\1_E\|_1 \right)^{1/2} \right).
\end{align*}
Thus, by Step 2,
$$
A[\rho] \leq \|\rho\|_1^{-1} \left( 2^{-1} (c_{N,\delta}')^{-1/2}\, \mathcal D[\rho]^{1/2} + c_{N,\delta}^{-1/2} \left( \mathcal D[\rho] + (c_{N,\delta}')^{-1/2} \|\rho\|_1 \,\mathcal D[\rho]^{1/2} \right)^{1/2} \right).
$$
This implies, with a constant $c_{N,\delta}''>0$ depending only on $N$ and $\delta$,
$$
A[\rho] \leq (c_{N,\delta}'')^{-1/2} \|\rho\|_1^{-1}
\times
\begin{cases}
\mathcal D[\rho]^{1/2} & \text{if}\ \mathcal D[\rho]\geq \|\rho\|_1^2 \,,\\
\|\rho\|_1^{-1/2} \mathcal D[\rho]^{1/4} & \text{if}\ \mathcal D[\rho]< \|\rho\|_1^2 \,.
\end{cases}
$$
If $\mathcal D[\rho]\geq \|\rho\|_1^2$, this is the claimed inequality. If $\mathcal D[\rho]<\|\rho\|_1^2$ we use
$$
A[\rho] \geq \alpha_{N,\delta}^{1/2}\, A[\rho]^{1/2} \,,
$$
and obtain again the claimed inequality. This concludes the proof of the theorem.
\end{proof}

%%%%%%%%%%%%%%%%%

\section{Explicit spectral analysis}\label{sec:hessian}

Throughout this appendix, we assume $N\geq 2$. The paper \cite{Ch} contains a soft and quite general, but rather lengthy argument which shows that the quadratic form
$$
\mathcal Q[F] := \frac12 \iint_{\Sph^{N-1}\times\Sph^{N-1}} F(\omega) \1_{\{|\omega-\omega'|<\tilde R/R\}} F(\omega')\,d\omega\,d\omega'
$$
is bounded from above by
\begin{equation}
\label{eq:christhessianbound}
A\, \Gamma\, \|F\|_{L^2(\Sph^{N-1})}^2
\qquad\text{with a constant}\ A<\tfrac12
\end{equation}
for all $F\in L^2(\Sph^{N-1})$ satisfying
\begin{equation}
\label{eq:appfortho}
\int_{\Sph^{N-1}} F(\omega)\,d\omega = \int_{\Sph^{N-1}} \omega_n F(\omega)\,d\omega = 0 \,,
\qquad\text{for all}\ n=1,\ldots,N \,;
\end{equation}
see \cite[Sections 7 and 10]{Ch}. Here $R$ and $\tilde R$ denote the radii of $E^*$ and $B$, respectively, and the constant $\Gamma$ was defined in Lemma \ref{phi}. In this appendix we reprove the bound \eqref{eq:christhessianbound} by computing explicitly the eigenvalues of $Q$. The possibility of such a proof is mentioned in \cite[Section 8]{Ch}. This might be useful in applications where more precise information about these eigenvalues is needed.

For simplicity of notation we set $a:= \tilde R^2/(2R^2)$, so that
$$
\{ |\omega-\omega'|<\tilde R/R\} = \{ \omega\cdot\omega'> 1- a \} \,.
$$
The quadratic form $\mathcal Q$ is invariant under rotations and therefore is diagonal with respect to a decomposition into spherical harmonics. By the Funk--Hecke formula, the eigenvalue $\lambda_{N,\ell}$ of $\mathcal Q$ on spherical harmonics of degree $\ell$ is given by
$$
\lambda_{N,\ell} = \tfrac12 |\Sph^{N-2}| (C_\ell^{((N-2)/2)}(1))^{-1} \int_{1-a}^1 C_\ell^{((N-2)/2)}(t) (1-t^2)^{(N-3)/2}\,dt \,.
$$
Here $C_\ell^{(\alpha)}$ denote the Gegenbauer polynomials, see \cite[Chapter 22]{AbSt}. For $\ell=0$ we have $C_0^{(\alpha)}\equiv 1$ for all $\alpha\geq 0$, and therefore
$$
\lambda_{N,0} = \tfrac12 |\Sph^{N-2}| \int_{1-a}^1 (1-t^2)^{(N-3)/2}\,dt \,.
$$
We now compute $\lambda_{N,\ell}$ for $\ell\geq 1$. We write the above integral as the difference of an integral from 0 to 1 and one from 0 to $1-a$ and apply the formula \cite[(22.13.2)]{AbSt}. (When $N=2$, we first need to insert \cite[(22.5.4)]{AbSt} into \cite[(22.13.2)]{AbSt} to obtain a formula for the integral of $C_n^{(0)}$.) Then we use twice the formula \cite[(22.2.3)]{AbSt}
$$
C_n^{(\alpha)}(1) = \frac{(n+2\alpha-1)!}{n!\ (2\alpha-1)!} \quad\text{if}\ \alpha>0 \,,
\qquad
C_n^{(0)}(1) = \frac{2}{n} \quad\text{if}\ n\neq 0 \,,
$$
to simplify the resulting expression. We finally obtain
$$
\lambda_{N,\ell} = \frac{|\Sph^{N-2}|}{2(N-1)}\ \frac{C^{(N/2)}_{\ell-1}(1-a)}{C^{(N/2)}_{\ell-1}(1)} \, \left( 1 - (1-a)^2 \right)^{(N-1)/2} 
$$
This is the claimed formula for the eigenvalue.

Recall that $\Gamma$ was defined in Lemma \ref{phi}. We now show that
\begin{equation}
\label{eq:evtransl}
\tfrac12 \Gamma = \lambda_{N,1} \,.
\end{equation}
We have, in the sense of distributions,
$$
\nabla \phi(x) = \int_{E^*} \nabla_x \1_B(x-y)\,dy = - \int_{E^*} \nabla_y \1_B(x-y)\,dy = - \int_{\partial E^*} \nu_y \1_B(x-y)\,d\sigma(y) \,.
$$
Thus, for any $x=R\omega\in\partial E^*$ and introducing variables $y=R\omega'$ and $t=\omega\cdot\omega'$,
\begin{align*}
\phi'(R) & = - \int_{\partial E^*} \frac{x\cdot y}{|x|\,|y|} \1_B(x-y)\,d\sigma(y)
= - R^{N-1} \int_{\Sph^{N-1}} \omega\cdot\omega'\, \1_B(R(\omega-\omega'))\,d\omega' \\
& = - R^{N-1} |\Sph^{N-2}| \int_{1-a}^1 t (1-t^2)^{(N-3)/2}\,dt 
= - R^{N-1} |\Sph^{N-2}| \int_{|1-a|}^1 t (1-t^2)^{(N-3)/2}\,dt \\
& = - \tfrac12 R^{N-1} |\Sph^{N-2}| \int_{(1-a)^2}^1 (1-s)^{(N-3)/2}\,ds \\
& = - \frac1{N-1} R^{N-1} |\Sph^{N-2}| \left( 1- (1-a)^2 \right)^{(N-1)/2} = - 2\, R^{N-1} \lambda_{N,1} \,.
\end{align*}
In the last equality we used $C^{(\alpha)}_0\equiv 1$ together with our formula for $\lambda_{N,1}$. This proves \eqref{eq:evtransl}.

For $F\in L^2(\Sph^{N-1})$ we denote by $F_\ell$ its projection on the subspace of spherical harmonics of degree $\ell$ and then
$$
\mathcal Q[F] = \sum_{\ell=0}^\infty \lambda_{N,\ell} \|F_\ell\|^2
\qquad\text{and}\qquad
\|F\|^2 = \sum_{\ell=0}^\infty \|F_\ell\|^2 \,.
$$
The orthogonality conditions \eqref{eq:appfortho} imply that $F_0=0$ and $F_1=0$. Thus, in view of \eqref{eq:evtransl}, the bound \eqref{eq:christhessianbound} is equivalent to the bounds
\begin{equation}
\label{eq:christhessianbound1}
\lambda_{N,\ell} \leq 2A \,\lambda_{N,1} 
\qquad\text{for all}\ \ell\geq 2\ \text{with a constant}\ A<\tfrac12 \,.
\end{equation}
In terms of the explicit expression of $\lambda_{N,\ell}$, recalling that $C^{(\alpha)}_0\equiv 1$ we find
$$
\frac{\lambda_{N,\ell}}{\lambda_{N,1}} = \frac{C_{\ell-1}^{(N/2)}(1-a)}{C_{\ell-1}^{(N/2)}(1)} \,.
$$
Thus, from \cite[(22.14.2)]{AbSt} we infer that $\lambda_{N,\ell}\leq\lambda_{N,1}$ for all $\ell\geq 1$, but this is not quite enough to conclude \eqref{eq:christhessianbound1}. However, arguing as in \cite[Subsec.~7.33]{Sz} we see that the successive local maxima of $|C^{(N/2)}_{\ell-1}|$ in $[0,1]$  are strictly increasing. Moreover, since $|C_{\ell-1}^{(N/2)}|$ is even, we infer that $|C_{\ell-1}^{(N/2)}(t)|<C_{\ell-1}^{(N/2)}(1)$ for all $t\in(-1,1)$. This proves that, in fact, $\lambda_{N,\ell}<\lambda_{N,1}$ for all $\ell\geq 2$.

The uniform bound \eqref{eq:christhessianbound1} now follows from the fact that $\lambda_{N,\ell}\to 0$ as $\ell\to\infty$, which is a consequence of the compactness in $L^2(\Sph^{N-1})$ of the quadratic form $\mathcal Q$. Alternatively, one can use asymptotics of orthogonal polynomials \cite[Chapter 8]{Sz} to conclude that $\lambda_{N,\ell}/\lambda_{N,1} \to 0$ as $\ell\to\infty$. This concludes our alternative proof of \eqref{eq:christhessianbound}.

The last argument might sound like it involves compactness. This is not so, since it is easy to obtain an upper bound on $\lambda_{N,\ell}$ that tends to zero and therefore the condition $\lambda_{N,\ell}<\lambda_{N,1}$ needs only to be checked for an explicit finite number of $\ell$'s. Let us denote by $\mu_{n}$ the eigenvalues of the quadratic form $\mathcal Q$ in non-increasing order, \emph{repeated according to multiplicities}. (Thus, the $\mu_n$ are the same as the $\lambda_{N,\ell}$, up to repetitions.) Then, by computing the Hilbert--Schmidt norm of the operator associated with the quadratic form $\mathcal Q$ we find that
$$
\sum_n \mu_n^2 = \tfrac14 \iint_{\Sph^{N-1}\times\Sph^{N-1}} \1_{\{R|\omega-\omega'|<\tilde R\}} \,d\omega\,d\omega' \,.
$$
The integral on the right side equals
$$
\tfrac14 |\Sph^{N-1}| |\Sph^{N-2}| \int_{1-a}^1 (1-t^2)^{(N-3)/2} \,dt = \tfrac12 |\Sph^{N-1}| \, \lambda_{N,0} \,.
$$
Thus, for all $n$,
$$
\mu_n \leq n^{-1/2} \left( \tfrac12 |\Sph^{N-1}| \, \lambda_{N,0} \right)^{1/2} \,,
$$
and consequently for $n\geq 2 |\Sph^{N-1}| \, \lambda_{N,0} \lambda_{N,1}^{-2}=:n_0$, the inequality $\mu_n\leq (1/2)\lambda_{N,1}$ is satisfied. Therefore, \eqref{eq:christhessianbound1} holds with
$$
2A = \max\left\{ \frac 12, \max_{n<n_0\,,\, \mu_n\neq\lambda_{N,0}} \frac{\mu_n}{\lambda_{N,1}} \right\}.
$$
This formula also justifies the claim that the constant $A$ can be chosen to depend only on $N$ and $\delta$ (and not on $R$ and $\tilde R$).

\begin{remark}\label{threesets}
In the introduction we stated that the analysis of the Hessian in the present case is simpler than in Christ's case of three sets. Let us elaborate on this point. In our case we essentially needed to prove that certain eigenvalues $\lambda_{N,\ell}$ satisfy an inequality $\lambda_{N,\ell}<\lambda_{N,1}$ for $\ell\geq 2$. In the three set case the eigenvalues $\lambda_{N,\ell}$ (with different values of $a$) become entries in a $3\times 3$ matrix, whose eigenvalues are then of concern. For instance, the (1,2) entry of this matrix comes from the eigenvalues of the quadratic form with kernel $\1_{\{|r_1\omega+r_2\omega'|<r_3\}}$ for certain fixed $r_1,r_2,r_3>0$. Since
$$
\{|r_1\omega+r_2\omega'|<r_3\} = \{ \omega\cdot\omega' < b_{12}\}
$$
with $b_{12} := (r_3^2 - r_1^2 - r_2^2)/(2r_1r_2)$, the above computations yield a formula for the eigenvalues of this quadratic form. However, what is needed is an understanding of the eigenvalues of the resulting $3\times 3$ matrix. Another complication comes from additional zero modes due to an additional symmetry. Similar difficulties were overcome in Christ's proof of a quantitative stability theorem for Young's convolution inequality \cite{Ch2}. The Hessian there is treated explicitly in terms of Hermite functions, whose role is similar to that of spherical harmonics here.
\end{remark}

%%%%%%%%%%%%%%%%%%%%%%%%%%%%%%%%%%%%%%%%%%%%%%%%%%%%%%%%%%%%%%%%%%%%%%%%%%%%%%%%
%%%%%%%%%%%

\bibliographystyle{amsalpha}

\end{document}